\documentclass[a4paper,10pt]{amsart}

\usepackage[utf8]{inputenc}
\usepackage{url}

\usepackage[expansion=false]{microtype}
\usepackage{amssymb,amsmath,amsopn,amsxtra,amsthm,amsfonts}
\usepackage{xr}
\usepackage[mathcal]{euscript}
\usepackage{mathalfa}

\usepackage[english, status=final, nomargin, inline]{fixme}
\fxusetheme{color}

\usepackage[pdfusetitle,unicode,hidelinks]{hyperref}

\usepackage{enumitem}
\setlist[enumerate,1]{label={(\arabic*)},itemsep=\parskip} %,leftmargin=0pt
\setlist[itemize,1]{itemsep=\parskip} %,leftmargin=0pt
\newlist{thmlist}{enumerate}{2}
\setlist[thmlist,1]{label={\em(\roman*)},ref={(\roman*)},%
  itemsep=\parskip,leftmargin=*,align=left}
\setlist[thmlist,2]{label={\em(\alph*)},ref={(\alph*)},%
  itemsep=\parskip,leftmargin=*,align=left,topsep=0.1cm}
\newlist{remlist}{enumerate}{2}
\setlist[remlist,1]{label={(\roman*)},ref={(\roman*)},itemsep=\parskip,%
  leftmargin=*,align=left}
\setlist[remlist,2]{label={(\alph*)},ref={(\alph*)},itemsep=\parskip,%
  leftmargin=*,align=left,topsep=0.1cm}

\makeatletter
\let\c@equation\c@subsubsection

\makeatother

\newtheorem{cor}[subsubsection]{Corollary}
\newtheorem{lem}[subsubsection]{Lemma}
\newtheorem{prop}[subsubsection]{Proposition}

\newtheorem{thm}[subsubsection]{Theorem}

\newtheorem*{claim*}{Claim}

\theoremstyle{definition}

\newtheorem{rem}[subsubsection]{Remark}

\renewcommand{\eqref}[1]{(\ref{#1})}

\usepackage{tikz}
\usetikzlibrary{matrix}
\usepackage{tikz-cd}

\usepackage{xspace}

\usepackage{xifthen}

\usepackage{xparse}

\newcommand{\nc}{\newcommand}
\nc{\renc}{\renewcommand}
\nc{\ssec}{\subsection}
\nc{\sssec}{\subsubsection}
\nc{\on}{\operatorname}
\nc{\term}[1]{#1\xspace}

\DeclareMathSymbol{A}{\mathalpha}{operators}{`A}
\DeclareMathSymbol{B}{\mathalpha}{operators}{`B}
\DeclareMathSymbol{C}{\mathalpha}{operators}{`C}
\DeclareMathSymbol{D}{\mathalpha}{operators}{`D}
\DeclareMathSymbol{E}{\mathalpha}{operators}{`E}
\DeclareMathSymbol{F}{\mathalpha}{operators}{`F}
\DeclareMathSymbol{G}{\mathalpha}{operators}{`G}
\DeclareMathSymbol{H}{\mathalpha}{operators}{`H}
\DeclareMathSymbol{I}{\mathalpha}{operators}{`I}
\DeclareMathSymbol{J}{\mathalpha}{operators}{`J}
\DeclareMathSymbol{K}{\mathalpha}{operators}{`K}
\DeclareMathSymbol{L}{\mathalpha}{operators}{`L}
\DeclareMathSymbol{M}{\mathalpha}{operators}{`M}
\DeclareMathSymbol{N}{\mathalpha}{operators}{`N}
\DeclareMathSymbol{O}{\mathalpha}{operators}{`O}
\DeclareMathSymbol{P}{\mathalpha}{operators}{`P}
\DeclareMathSymbol{Q}{\mathalpha}{operators}{`Q}
\DeclareMathSymbol{R}{\mathalpha}{operators}{`R}
\DeclareMathSymbol{S}{\mathalpha}{operators}{`S}
\DeclareMathSymbol{T}{\mathalpha}{operators}{`T}
\DeclareMathSymbol{U}{\mathalpha}{operators}{`U}
\DeclareMathSymbol{V}{\mathalpha}{operators}{`V}
\DeclareMathSymbol{W}{\mathalpha}{operators}{`W}
\DeclareMathSymbol{X}{\mathalpha}{operators}{`X}
\DeclareMathSymbol{Y}{\mathalpha}{operators}{`Y}
\DeclareMathSymbol{Z}{\mathalpha}{operators}{`Z}

\nc{\sA}{\ensuremath{\mathcal{A}}\xspace}
\nc{\sB}{\ensuremath{\mathcal{B}}\xspace}
\nc{\sC}{\ensuremath{\mathcal{C}}\xspace}
\nc{\sD}{\ensuremath{\mathcal{D}}\xspace}
\nc{\sE}{\ensuremath{\mathcal{E}}\xspace}
\nc{\sF}{\ensuremath{\mathcal{F}}\xspace}
\nc{\sG}{\ensuremath{\mathcal{G}}\xspace}
\nc{\sH}{\ensuremath{\mathcal{H}}\xspace}
\nc{\sI}{\ensuremath{\mathcal{I}}\xspace}
\nc{\sJ}{\ensuremath{\mathcal{J}}\xspace}
\nc{\sK}{\ensuremath{\mathcal{K}}\xspace}
\nc{\sL}{\ensuremath{\mathcal{L}}\xspace}
\nc{\sM}{\ensuremath{\mathcal{M}}\xspace}
\nc{\sN}{\ensuremath{\mathcal{N}}\xspace}
\nc{\sO}{\ensuremath{\mathcal{O}}\xspace}
\nc{\sP}{\ensuremath{\mathcal{P}}\xspace}
\nc{\sQ}{\ensuremath{\mathcal{Q}}\xspace}
\nc{\sR}{\ensuremath{\mathcal{R}}\xspace}
\nc{\sS}{\ensuremath{\mathcal{S}}\xspace}
\nc{\sT}{\ensuremath{\mathcal{T}}\xspace}
\nc{\sU}{\ensuremath{\mathcal{U}}\xspace}
\nc{\sV}{\ensuremath{\mathcal{V}}\xspace}
\nc{\sW}{\ensuremath{\mathcal{W}}\xspace}
\nc{\sX}{\ensuremath{\mathcal{X}}\xspace}
\nc{\sY}{\ensuremath{\mathcal{Y}}\xspace}
\nc{\sZ}{\ensuremath{\mathcal{Z}}\xspace}

\nc{\bA}{\ensuremath{\mathbf{A}}\xspace}
\nc{\bB}{\ensuremath{\mathbf{B}}\xspace}
\nc{\bC}{\ensuremath{\mathbf{C}}\xspace}
\nc{\bD}{\ensuremath{\mathbf{D}}\xspace}
\nc{\bE}{\ensuremath{\mathbf{E}}\xspace}
\nc{\bF}{\ensuremath{\mathbf{F}}\xspace}
\nc{\bG}{\ensuremath{\mathbf{G}}\xspace}
\nc{\bH}{\ensuremath{\mathbf{H}}\xspace}
\nc{\bI}{\ensuremath{\mathbf{I}}\xspace}
\nc{\bJ}{\ensuremath{\mathbf{J}}\xspace}
\nc{\bK}{\ensuremath{\mathbf{K}}\xspace}
\nc{\bL}{\ensuremath{\mathbf{L}}\xspace}
\nc{\bM}{\ensuremath{\mathbf{M}}\xspace}
\nc{\bN}{\ensuremath{\mathbf{N}}\xspace}
\nc{\bO}{\ensuremath{\mathbf{O}}\xspace}
\nc{\bP}{\ensuremath{\mathbf{P}}\xspace}
\nc{\bQ}{\ensuremath{\mathbf{Q}}\xspace}
\nc{\bR}{\ensuremath{\mathbf{R}}\xspace}
\nc{\bS}{\ensuremath{\mathbf{S}}\xspace}
\nc{\bT}{\ensuremath{\mathbf{T}}\xspace}
\nc{\bU}{\ensuremath{\mathbf{U}}\xspace}
\nc{\bV}{\ensuremath{\mathbf{V}}\xspace}
\nc{\bW}{\ensuremath{\mathbf{W}}\xspace}
\nc{\bX}{\ensuremath{\mathbf{X}}\xspace}
\nc{\bY}{\ensuremath{\mathbf{Y}}\xspace}
\nc{\bZ}{\ensuremath{\mathbf{Z}}\xspace}

\nc{\dA}{\ensuremath{\mathds{A}}\xspace}
\nc{\dB}{\ensuremath{\mathds{B}}\xspace}
\nc{\dC}{\ensuremath{\mathds{C}}\xspace}
\nc{\dD}{\ensuremath{\mathds{D}}\xspace}
\nc{\dE}{\ensuremath{\mathds{E}}\xspace}
\nc{\dF}{\ensuremath{\mathds{F}}\xspace}
\nc{\dG}{\ensuremath{\mathds{G}}\xspace}
\nc{\dH}{\ensuremath{\mathds{H}}\xspace}
\nc{\dI}{\ensuremath{\mathds{I}}\xspace}
\nc{\dJ}{\ensuremath{\mathds{J}}\xspace}
\nc{\dK}{\ensuremath{\mathds{K}}\xspace}
\nc{\dL}{\ensuremath{\mathds{L}}\xspace}
\nc{\dM}{\ensuremath{\mathds{M}}\xspace}
\nc{\dN}{\ensuremath{\mathds{N}}\xspace}
\nc{\dO}{\ensuremath{\mathds{O}}\xspace}
\nc{\dP}{\ensuremath{\mathds{P}}\xspace}
\nc{\dQ}{\ensuremath{\mathds{Q}}\xspace}
\nc{\dR}{\ensuremath{\mathds{R}}\xspace}
\nc{\dS}{\ensuremath{\mathds{S}}\xspace}
\nc{\dT}{\ensuremath{\mathds{T}}\xspace}
\nc{\dU}{\ensuremath{\mathds{U}}\xspace}
\nc{\dV}{\ensuremath{\mathds{V}}\xspace}
\nc{\dW}{\ensuremath{\mathds{W}}\xspace}
\nc{\dX}{\ensuremath{\mathds{X}}\xspace}
\nc{\dY}{\ensuremath{\mathds{Y}}\xspace}
\nc{\dZ}{\ensuremath{\mathds{Z}}\xspace}

\nc{\bbA}{\ensuremath{\mathbb{A}}\xspace}
\nc{\bbB}{\ensuremath{\mathbb{B}}\xspace}
\nc{\bbC}{\ensuremath{\mathbb{C}}\xspace}
\nc{\bbD}{\ensuremath{\mathbb{D}}\xspace}
\nc{\bbE}{\ensuremath{\mathbb{E}}\xspace}
\nc{\bbF}{\ensuremath{\mathbb{F}}\xspace}
\nc{\bbG}{\ensuremath{\mathbb{G}}\xspace}
\nc{\bbH}{\ensuremath{\mathbb{H}}\xspace}
\nc{\bbI}{\ensuremath{\mathbb{I}}\xspace}
\nc{\bbJ}{\ensuremath{\mathbb{J}}\xspace}
\nc{\bbK}{\ensuremath{\mathbb{K}}\xspace}
\nc{\bbL}{\ensuremath{\mathbb{L}}\xspace}
\nc{\bbM}{\ensuremath{\mathbb{M}}\xspace}
\nc{\bbN}{\ensuremath{\mathbb{N}}\xspace}
\nc{\bbO}{\ensuremath{\mathbb{O}}\xspace}
\nc{\bbP}{\ensuremath{\mathbb{P}}\xspace}
\nc{\bbQ}{\ensuremath{\mathbb{Q}}\xspace}
\nc{\bbR}{\ensuremath{\mathbb{R}}\xspace}
\nc{\bbS}{\ensuremath{\mathbb{S}}\xspace}
\nc{\bbT}{\ensuremath{\mathbb{T}}\xspace}
\nc{\bbU}{\ensuremath{\mathbb{U}}\xspace}
\nc{\bbV}{\ensuremath{\mathbb{V}}\xspace}
\nc{\bbW}{\ensuremath{\mathbb{W}}\xspace}
\nc{\bbX}{\ensuremath{\mathbb{X}}\xspace}
\nc{\bbY}{\ensuremath{\mathbb{Y}}\xspace}
\nc{\bbZ}{\ensuremath{\mathbb{Z}}\xspace}

\nc{\mrm}[1]{\ensuremath{\mathrm{#1}}\xspace}
\nc{\mbf}[1]{\ensuremath{\mathbf{#1}}\xspace}
\nc{\mcal}[1]{\ensuremath{\mathcal{#1}}\xspace}
\nc{\msc}[1]{\ensuremath{\mathscr{#1}}\xspace}

\renc{\bar}[1]{\overline{#1}}

\nc{\sub}{\subset}
\nc{\too}{\longrightarrow}
\nc{\hook}{\hookrightarrow}
\nc*{\hooklongrightarrow}{\ensuremath{\lhook\joinrel\relbar\joinrel\rightarrow}}
\nc{\hooklong}{\hooklongrightarrow}
\nc{\twoheadlongrightarrow}{\relbar\joinrel\twoheadrightarrow}
\nc{\shiso}{\approx}
\nc{\isoto}{\xrightarrow{\sim}}
\nc{\isofrom}{\xleftarrow{\sim}}
\renc{\ge}{\geqslant}
\renc{\le}{\leqslant}
\renc{\geq}{\geqslant}
\renc{\leq}{\leqslant}

\nc{\id}{\mathrm{id}}

\DeclareMathOperator{\Hom}{\mathrm{Hom}}
\nc{\uHom}{\underline{\smash{\Hom}}}
\DeclareMathOperator{\Maps}{\mathrm{Maps}}

\DeclareMathOperator{\End}{\mathrm{End}}

\nc{\Pre}{\mathrm{PSh}{}}
\nc{\Shv}{\mathrm{Shv}{}}
\nc{\uEnd}{\underline{\smash{\End}}}

\renc{\lim}{\operatorname*{lim}}
\nc{\colim}{\operatorname*{colim}}
\nc{\Cofib}{\on{Cofib}}
\nc{\Fib}{\on{Fib}}
\nc{\initial}{\varnothing}
\nc{\op}{\mathrm{op}}

\renc{\coprod}{\sqcup}

\nc{\bDelta}{\mbf{\Delta}}
\nc{\DM}{\mbf{DM}}
\nc{\eff}{\mathrm{eff}}
\nc{\veff}{\mathrm{veff}}
\nc{\cyc}{{\mrm{cyc}}}
\nc{\corr}{{\on{corr}}}
\nc{\ft}{\mrm{ft}}
\nc{\flf}{\mrm{flf}}
\nc{\fet}{{\mrm{f\acute et}}}
\nc{\fsyn}{{\mrm{fsyn}}}
\nc{\syn}{{\mrm{syn}}}
\nc{\lci}{{\mrm{lci}}}
\nc{\Perf}{\mbf{Perf}}
\nc{\perf}{\mrm{perf}}
\nc{\oblv}{\mrm{oblv}}
\nc{\exact}{\on{exact}}

\nc{\F}{{\on{F}}}
\nc{\clopen}{{\mrm{clopen}}}
\nc{\B}{\mrm{B}}
\nc{\D}{\mrm{D}}
\nc{\Fin}{\on{Fin}}
\nc{\fin}{\mrm{fin}}
\nc{\Cut}{\on{Cut}}
\nc{\Cart}{\on{Cart}}
\nc{\pairs}{\mathsf{pairs}}
\nc{\Pairs}{\mathrm{Pair}}
\nc{\Trip}{\mathrm{Trip}}
\nc{\Lab}{\mathrm{Lab}}
\nc{\SL}{\mathrm{SL}}
\nc{\coCart}{\mathrm{coCart}}
\nc{\RKE}{\mathrm{RKE}}
\nc{\strict}{\mathrm{strict}}
\nc{\Emb}{\mathrm{Emb}}
\nc{\Split}{\mathrm{Split}}
\nc{\Set}{\mathrm{Set}}
\nc{\sSets}{\mathrm{sSets}}
\nc{\pb}{\mathrm{pb}}
\nc{\fib}{\mathrm{fib}}
\nc{\diff}{\mrm{diff}}
\nc{\gp}{\mrm{gp}}
\nc{\map}{\mrm{map}}

\nc{\mgp}{\mrm{mot-gp}}
\nc{\FSyn}{\mrm{FSyn}}
\nc{\FEt}{\mrm{FEt}}
\nc{\Spc}{\mrm{Spc}}
\nc{\Ob}{\mrm{Ob}}
\nc{\Spt}{\mrm{Spt}}
\nc{\T}{\bT}
\nc{\suspinf}{\Sigma^\infty}
\nc{\h}{\mrm{h}}
\nc{\uhom}{\underline{\mathrm{Hom}}}
\nc{\umap}{\underline{\mathrm{Maps}}}
\renc{\H}{\bH}
\nc{\Einfty}{{\sE_\infty}}
\nc{\Eone}{{\sE_1}}
\nc{\Stab}{\mrm{Stab}}
\nc{\lax}{{\mrm{lax}}}
\nc{\cocart}{{\mrm{cocart}}}
\nc{\Sch}{\on{Sch}}
\nc{\Fr}{\on{Fr}}
\nc{\A}{\mathbf{A}}
\nc{\N}{\mathbf{N}}
\nc{\Z}{\mathbf{Z}}
\nc{\Q}{\mathbf{Q}}
\nc{\Oo}{\mathcal{O}} 
\nc{\Fscr}{\mathcal{F}}
\nc{\Gscr}{\mathcal{G}}
\nc{\Ll}{\mathcal{L}} 
\nc{\Mm}{\mathcal{M}} 
\nc{\mm}{\mathrm{m}} 
\nc{\K}{\mrm{K}} 
\nc{\W}{\mrm{W}} 
\nc{\red}{{\on{red}}}
\nc{\Voev}{{\on{Voev}}}
\nc{\Corr}{\mrm{Corr}}
\nc{\Span}{\mathbf{Corr}}
\nc{\Gap}{\mrm{Gap}}
\nc{\Corrfr}{\Corr^{\fr}}
\nc{\Corrvfr}{\Corr^{\Vfr}}
\nc{\Spec}{\on{Spec}}
\nc{\Sm}{\on{Sm}}
\nc{\Gm}{\mathbf{G}_{\on{m}}}
\renc{\P}{\bP}
\nc{\nis}{\mathrm{nis}}
\nc{\zar}{\mathrm{zar}}
\nc{\et}{\mathrm{\acute et}}
\nc{\all}{\mathrm{all}}
\nc{\fold}{\mathrm{fold}}
\nc{\Fun}{\mathrm{Fun}}
\nc{\Ho}{\mathrm{Ho}}
\nc{\Segal}{\mathrm{Segal}}
\nc{\Mon}{\mrm{Mon}{}}
\nc{\Ab}{\mrm{Ab}}
\nc{\Sh}{\on{Sh}}
\nc{\M}{\mrm{M}}
\nc{\Lhtp}{L_{\A^1}}
\nc{\Lmot}{L_{\mrm{mot}}}
\nc{\mot}{\mrm{mot}}
\nc{\SH}{\mbf{SH}}
\nc{\RR}{\mbf{R}}
\nc{\CC}{\mbf{C}}
\nc{\Mod}{\mbf{Mod}}
\nc{\QCoh}{\mbf{QCoh}}
\nc{\MonUnit}{\mbf{1}}
\nc{\1}{\mbf{1}}
\nc{\tr}{\on{tr}}
\nc{\cotr}{\mrm{cotr}}
\nc{\vop}{\mrm{vop}}
\nc{\fr}{{\on{fr}}}
\nc{\Ar}{\mrm{Ar}}
\nc{\Vfr}{\on{Vfr}}
\nc{\frdiff}{{\on{frdiff}}}
\nc{\frGys}{\on{frGys}}
\nc{\SHfr}{\SH^{\fr}}
\nc{\SHfrdiff}{\SH^{\frdiff}}
\nc{\SHfrGys}{\SH^{\frGys}}
\nc{\InftyCat}{(\mathrm{\infty,1)\textnormal{-}Cat}}
\nc{\TriCat}{\mathrm{TriCat}}
\nc{\oneCat}{\mathrm{1\textnormal{-}Cat}}
\nc{\Cat}{\mathrm{Cat}}
\nc{\Th}{\on{Th}}

\nc{\CMon}{\mrm{CMon}{}}
\nc{\CAlg}{\mrm{CAlg}{}}
\nc{\MGL}{\mrm{MGL}}
\nc{\Seg}{\mrm{Seg}{}}
\nc{\GW}{\mrm{GW}{}}
\nc{\Tw}{\mrm{Tw}}
\nc{\sslash}{/\mkern-6mu/}
\nc{\PrL}{\mrm{Pr}^\mrm{L}}
\nc{\PrR}{\mrm{Pr}^\mrm{R}}
\nc{\pr}{\mrm{pr}}
\let\phi\varphi

\nc\efr{\mrm{efr}}
\nc\nfr{\mrm{nfr}}
\nc\dfr{\mrm{fr}}
\nc\tfr{\mrm{tfr}}
\nc\Vect{\mrm{Vect}}
\nc\sVect{\mrm{sVect}}
\nc{\fix}{\mrm{fix}}
\nc{\ho}{\mrm{h}}
\nc\Mfd{\mrm{Mfd}}
\nc{\PSh}{\mrm{PSh}}
\nc{\hzmw}{H \tilde\Z{}}
\nc{\Cor}{\mrm{Cor}{}}
\nc{\cormw}{\mrm{\widetilde{Cor}}{}}
\nc{\Chw}{\mrm{\widetilde{CH}}{}}
\nc{\Ex}{\mrm{Ex}}
\nc{\BM}{\mrm{BM}}

\nc{\Pic}{\mrm{Pic}}
\nc{\Br}{\mrm{Br}}
\nc{\pur}{\mathfrak p}
\nc{\angles}[1]{\langle #1\rangle}
\nc{\inv}[1]{[\tfrac{1}{#1}]}
\nc{\pinv}{\inv{p}}
\nc{\cinv}{\inv{p}}
\nc{\Sph}{\on{Sph}}
\nc{\KGL}{\mrm{KGL}}
\nc{\KH}{\mrm{KH}}
\nc{\Flag}{\mrm{Flag}}
\nc{\Pro}{\mrm{Pro}}
\nc{\Frac}{\mrm{Frac}}

\nc{\arc}{\mrm{arc}}
\nc{\rarc}{\mrm{rarc}}
\nc{\cdarc}{\mrm{cdarc}}
\nc{\vv}{\mrm{v}}
\nc{\rv}{\mrm{rv}}
\nc{\cdv}{\mrm{cdv}}
\nc{\hh}{\mrm{h}}
\nc{\cdh}{\mrm{cdh}}
\nc{\rh}{\mathrm{rh}}
\nc{\Et}{\mathrm{Et}}
\nc{\Nis}{\mathrm{Nis}}
\nc{\Zar}{\mathrm{Zar}}
\nc{\cdp}{\mathrm{cdp}}

\nc{\RZ}{\mathrm{RZ}}
\nc{\qcqs}{\mathrm{qcqs}}
\nc{\aff}{\mathrm{aff}}
\nc{\cl}{\mathrm{cl}}
\nc{\Val}{\mathrm{Val}}
\nc{\GFin}{\mathrm{GFin}{}}
\nc{\Proj}{\mathrm{Proj}}

\nc{\inftyCat}{\term{$\infty$-category}}
\nc{\inftyCats}{\term{$\infty$-categories}}

\nc{\inftyOneCat}{\term{$(\infty,1)$-category}}
\nc{\inftyOneCats}{\term{$(\infty,1)$-categories}}

\nc{\inftyGrpd}{\term{$\infty$-groupoid}}
\nc{\inftyGrpds}{\term{$\infty$-groupoids}}

\nc{\inftyTop}{\term{$\infty$-topos}}
\nc{\inftyTops}{\term{$\infty$-toposes}}

\nc{\inftyTwoCat}{\term{$(\infty,2)$-category}}
\nc{\inftyTwoCats}{\term{$(\infty,2)$-categories}}

\title{}
\title{THH and TC are (very) far from being homotopy functors}

\author[E. Elmanto]{Elden Elmanto}
\address{Department of Mathematics\\
Harvard University\\
1 Oxford St.\
Cambridge, MA 02138\\
USA}
\email{\href{mailto:elmanto@math.harvard.edu}{elmanto@math.harvard.edu}}
\urladdr{\url{https://www.eldenelmanto.com/}}

\date{\today}

\begin{document}

\def\comp{\wedge}

\maketitle

\begin{abstract} We compute the $\bA^1$-localization of several invariants of schemes namely, topological Hochschild homology ($THH$), topological cyclic homology ($TC$) and topological periodic cyclic homology ($TP$). This procedure is quite brutal and kills the completed versions of most of these invariants. The main ingredient for the vanishing statements is the vanishing of $\bA^1$-localization of de Rham cohomology (and, eventually, crystalline cohomology) in positive characteristics.
\end{abstract}

%In general, this procedure is quite brutal: it kills $THH$ in all characteristic (generalizing a result of Geller and Weibel), and identifies $TP$ with negative cyclic homology ($TC^-$). We further prove that $TC$ and $TP$ are killed in characteristic positive characteristics and give a mixed characteristic result for $TC$. 

\section{Introduction and vista}

In this short paper, we compute the $\bA^1$-localization of several invariants relevant to recent developments in $p$-adic Hodge theory. They turn out to be mostly zero. These are ``no-go theorems" which state that, at least in characteristic $p$, the motivic perspective of Morel and Voevodsky \cite{MV}, which is based on the notion of $\bA^1$-invariance, is incompatible \emph{in a strong way} with the motivic perspective of \cite{BMS2}, which is based on descent properties of ``trace invariants" (e.g. topological Hochschild and cyclic homology) of schemes.

This is perhaps unsurprising: Ayoub in \cite[Lemme 3.10]{AyoubEtale} and Cisinski-D\'eglise in \cite[Proposition A.3.1]{CDetale} have proved that versions of \'etale motives are $\bZ[\tfrac{1}{p}]$-linear while Bachmann and Hoyois have upgraded these to a spectral version (unpublished). In other words, the stable $\infty$-category of $p$-complete \'etale motives on a characteristic $p$-scheme vanishes. On the other hand the aforementioned ``trace invariants" are $p$-adic \'etale sheaves. This paper adds another collection of results along these lines (though, \emph{a priori} unrelated).

The key actor among these ``trace invariants" is topological cyclic homology $TC$ of \cite{bhm}, recently revisited by Nikolaus and Scholze \cite{nikolaus-scholze}. This theory is not, in general, $\bA^1$-invariant. In fact, as explained in Remark~\ref{rem:tc}, it \emph{must} not be in for it to be of any use for $K$-theory. We prove:
\begin{thm} $L_{\bA^1}TC$ vanishes after profinite completion. In other words, $L_{\bA^1}TC$ is purely rational.
\end{thm}
This is stated more precisely in Theorem~\ref{thm:vanish}, and says that the profinitely completed version of this invariant vanishes after $\bA^1$-localization. One should contrast this to the situation for algebraic $K$-theory which is also not $\bA^1$-invariant in general but whose $\bA^1$-localization, Weibel's $KH$-theory \cite{Weibel}, is still a vastly interesting invariant (in any characteristic).

\subsection{Summary} The title of this paper is an homage to the results of Geller and Weibel \cite{geller-weibel} in characteristic zero. We give a reformulation of what they did in Theorem~\ref{lem:hh}, which immediately adapts to the contractibility statement for $L_{\bA^1}THH$ in Theorem~\ref{lem:thh}. We then proceed to compute these invariants in characteristic $p$ and later over $\bZ$. 

One would like to prove that the $\bA^1$-localization of $TP$ is zero. The problem, and this is ultimately the technical crux of the present paper, is that the Tate construction does not commute with colimits in general.

%\begin{exam}[Shah] We learned from Jay Shah the following example. Let $\rho:S^0 \rightarrow S^{\sigma}$ be the Euler class in equivariant homotopy theory (inclusion of two points into the 1-point compactified sign representation). We consider the Borelification functor $j^*:\Spt^{C_2} \rightarrow \Spt^{BC_2}$ and take the colimit
%\[
%j^*(S^0) \xrightarrow{j^*\rho} j^*(S^{\sigma}) \xrightarrow{j^*\rho}  j^*(S^{2\sigma}) \cdots \rightarrow j^*(S^{n\sigma}) \rightarrow \cdots.
%\]
%We claim that $(-)^{tC_2}:\Spt^{BC_2} \rightarrow \Spt$ does not preserve this. Indeed, on the one hand
%\[
%\colim (j^*(S^{n\sigma}))^{tC_2} \simeq (S^0)^{tC_2}
%\]
%since $(\rho)^{tC_2}$ is invertible. Hence this colimit is equivalent to the $2$-complete sphere by the Segal conjecture \cite{lin}. On the other hand, let $i^*:=\Phi^{C_2}: \Spt^{C_2} \rightleftarrows \Spt: i_*$ be the adjunction whose left adjoint is the geometric fixed points functor. Hence:
%\[
%\colim j^*(S^{n\sigma}) \simeq j^*(\colim S^{n\sigma}) \simeq j^*(i_*i^*S^{0})
%\]
%where the last equivalence is given by the well-known fact that geometric fixed points in this case is given by inverting $\rho$. But $j^*i_* \simeq 0$ hence the answer is zero which remains zero after taking $(-)^{tC_2}$.
%\end{exam}

 Therefore, the result for $THH$ does not immediately boostrap to $TP$. Instead we exploit the Bhatt-Morrow-Scholze (BMS) filtration \cite{BMS2} and reduce the problem to vanishing statements for $\bA^1$-localization of crystalline (Proposition~\ref{lem:contractible-1}) and, eventually, de Rham cohomology (Lemma~\ref{lem:contractible}). This gives us the vanishing result for $\bA^1$-localization of $TC$ first (Theorem~\ref{thm:tc-zero}) and, later, of TP (Corollary~\ref{lem:tp-zero}). Eventually, we prove profinite vanishing of $L_{\bA^1}TC$ over the integers (Theorem~\ref{thm:vanish}) using the previous results and the observation that $L_{\bA^1}TC$ is a truncating invariant.

\subsection{Vista} This paper is a ``\emph{in vitro} experiment" about the interaction between $\bA^1$-invariance and invariants derived from $THH$: the results, as one can see, are negative. However we do believe that these two perspectives are complementary and can be useful as long as they are not mixed together. In other words, we should separate them.

One concrete way in which ``separating them" is a good idea is the following cartesian square which lets us break down algebraic $K$-theory into constituent pieces (when restricted to noetherian schemes of finite dimension):
\begin{equation} \label{eq:morrow}
\begin{tikzcd}
K \ar{r} \ar{d} & TC \ar{d}\\
L_{\bA^1}K \ar{r} & L_{\mrm{cdh}}TC.
\end{tikzcd}
\end{equation}
Here $L_{\mrm{cdh}}$ is the sheafification functor with respect to Suslin and Voevodsky's cdh topology \cite{sv-chow}. The cartesian-ness of this square is deep: it requires knowing that 1) $L_{\bA^1}K \simeq L_{\mrm{cdh}}K$ which was first proved by Haesemeyer in characteristic zero \cite{Haesemeyer} and by Kerz-Strunk-Tamme in \cite[Theorem 6.3]{KST} in general, and 2) that $K^{\mrm{inv}}$ is a cdh sheaf, which is a result of \cite{gh-negative} over perfect fields and assuming resolution of singularities, and \cite{LT} in general. This last result ultimately boils down to the celebrated theorem of Dundas-Goodwillie-McCarthy \cite{dgm}. 

We see that the key idea is to find a bridge between the homotopy invariant and the trace perspectives --- in this case this is given by cdh-sheafification. We are currently conducting further investigations in and around the square~\eqref{eq:morrow}.

\subsection{Convention} We use standard $\infty$-categorical terminology. A little note on possible confusion: our functor $HH$ is the derived version, i.e., left Kan extended from polynomial algebras. Hence, so are the functors $HC^-, HP, HC$ and so on. Any kind of ``affine line" appearing in this note is the flat affine line so that $\pi_*(R[t]) = \pi_*(R)[t]$, not that there could be another option anyway since we will work with simplicial commutative rings as opposed to $\bE_{\infty}$-ring spectra.

\subsection{Acknowledgements} I would like to thank Joseph Ayoub, Akhil Mathew, Matthew Morrow, Jay Shah, Zijian Yao and Allen Yuan for useful conversations, Benjamin Antieau, Sanath Devalpurkar, Jeremiah Heller, Arpon Raksit, Chuck Weibel and an anonymous referee on comments on an earlier draft and Lars Hesselholt who suggested that the vanishing result over the integers should be true. I would also like to thank Benjamin Antieau, Tom Bachmann, Lars Hesselholt, Marc Hoyois and Matthew Morrow for informing my perspective on ``motives" over the years. Lastly I would like to thank Vitoria the cat for constant distractions.

\section{Topological Hochschild Homology}
\sssec{} Let $R$ be a base (animated) ring and $\CAlg_R$ denote the $\infty$-category of animated $R$-algebras, concretely presented as the $\infty$-category obtained from simplicial commutative $R$-algebras and inverting weak equivalences or the sifted-colimit completion of polynomial $R$-algebras. We have the exact localization endofunctor
\[
L_{\bA^1}:\PSh(\CAlg_R^{\op},\Spt) \rightarrow \PSh(\CAlg_R^{\op},\Spt),
\]
reflecting presheaves of spectra onto $\bA^1$-invariant presheaves: those that convert the canonical map $S \rightarrow S[t]$ to equivalences for all $S \in \CAlg_R$. Also note that since homotopy invariant presheaves are stable under colimits, the endofunctor $L_{\bA^1}$ preserves colimits. A concrete formula for this functor is given object-wise by the formula\footnote{We learned from Weibel some history behind this formula and we take this opportunity to record it. On $\pi_0$, this was introduced by Swan and Gersten in \cite{gersten-h}. The simplicial ring $\bZ[\Delta^{\bullet}]$ was then considered by D. Anderson in \cite{anderson-k}. Weibel took the conceptual leap of taking the geometric realization of $K(R[\Delta^{\bullet}])$ in spectra to construct his $KH$-theory. In his Luminy talk, Suslin then suggested this construction as a recipe for motivic cohomology, details are in a joint paper with Voevodsky \cite{Suslin:1996}. Of course, this construction has since been central to workers in $\bA^1$-homotopy theory as a formula for $\bA^1$-localization, beginning with the introduction of the subject \cite{MV}.}:
\begin{equation} \label{eq:suslin}
(L_{\bA^1}\sF)(S) = \colim_{\Delta^{\op}} \sF(S[\Delta^{\bullet}]).
\end{equation}
One consequence of~\eqref{eq:suslin} is that $L_{\bA^1}$ is strong symmetric monoidal since the (pointwise) symmetric monoidal structure on presheaves of spectra commutes with colimits and the colimit appearing in~\eqref{eq:suslin} is sifted. Therefore $L_{\bA^1}$ preserves algebras and modules over them. 

\ssec{Hochschild homology} We now present the main result of Geller and Weibel's paper \cite[Theorem 2.1]{geller-weibel}; it is morally the same proof as theirs.

\begin{thm}[Geller-Weibel]\label{lem:hh} $L_{\bA^1}$ of $HH, HC$ are zero, while the canonical map $HC^- \rightarrow HP$ is an $L_{\bA^1}$-equivalence.
\end{thm}

\begin{proof} The second statement follows from assertion about $HC$ by the norm-cofiber sequence:
\[
\Sigma HC \rightarrow HC^- \xrightarrow{\mrm{can}} HP,
\]
 and the fact that $L_{\bA^1}$ is exact.

To prove the assertion about $HC$, we first prove the assertion about $HH$. The homotopy groups of the spectrum\footnote{Actually, in this case, an object of the derived $\infty$-category of abelian groups.} $(L_{\bA^1}HH)(S)$ are modules over the ring
\[
\pi_0((L_{\bA^1}HH)(S) \simeq \mrm{Coeq}(\pi_0(HH(S[t])) = \pi_0(S)[t] \rightrightarrows \pi_0(HH(S))=\pi_0(S)).
\] This ring is actually zero since the coequalizer instructs us to set $t=0=1$. Therefore since the homotopy groups of $L_{\bA^1}HH(S)$ are modules over the zero ring, they are all zero and hence the spectrum itself is contractible.

Now, we need to show that
\[
L_{\bA^1}(HC(-)) \simeq 0.
\]
Evaluating this on a $S \in \CAlg_R$, we need to compute the geometric realization of the simplicial object 
\[
n \mapsto HH(S[\Delta^{n}])_{hS^1}.
\] To compute this, note that the diagram 
\[
n \mapsto HH(S[\Delta^n])
\]
is a diagram in $S^1$-spectra since the transition maps are induced by ring maps, whence the geometric realization is an $S^1$-spectrum. Therefore from the fact that geometric realizations and taking $S^1$-orbits commute (they are both colimits) we get that 
\[
|HH(S[\Delta^{\bullet}])_{hS^1}| \simeq |HH(S[\Delta^{\bullet}])|_{hS^1} \simeq |0|_{hS^1} \simeq 0.
\]
\qedhere\end{proof}

\sssec{} Suppose for a moment that $R = \bQ$, hence we are looking at animated rings in characteristic zero. In this situation a result of Kassel \cite[Corollary 3.12]{kassel} proves that 
\[
L_{\bA^1}HP \simeq HP,
\] by way of comparison with de Rham cohomology (see also \cite[Theorem III.5.1]{goodwillie-loopspace} for a direct proof of a more general ``homotopy invariance statement"). As a result we deduce \cite[Theorem 4.1]{geller-weibel}:
\begin{cor}[Geller-Weibel]\label{cor:gw} Let $R$ be a ring of characteristic zero. Then we have a canonical equivalence 
\[
(L_{\bA^1}HC^-)(R) \simeq HP(R).
\]
\end{cor}

\begin{rem} \label{rem:assoc} The results of \cite{geller-weibel} also works in the generality of non-commutative rings. We do recover their results in this generality as well; see Remark~\ref{rem:ncworld}.
\end{rem}

\ssec{Topological Hochschild homology} The story for $THH$ is similar and the proof of Theorem~\ref{lem:hh} goes through in the topological setting.

\begin{thm} \label{lem:thh} $L_{\bA^1}$ of $THH$ and $THH_{hS^1}$ are zero. On the other hand, the canonical map $\mrm{can}:TC^- \rightarrow TP$ is an $L_{\bA^1}$-equivalence.
\end{thm}

\begin{proof} Indeed, the key observation of Theorem~\ref{lem:hh} is that $\pi_0(\bL_{\bA^1}HH(S)) \simeq 0$. But then $\pi_0(THH(S)) \simeq \pi_0(HH(S)) \simeq \pi_0(S)$ and the same argument follows through.
\end{proof}

\section{Characteristic $p$}  \label{sect:charp}

Suppose now that $R$ is an animated ring of characteristic $p > 0$ (in other words $R \in \CAlg_{\bF_p}$). 

\sssec{} We begin with a discussion of $L_{\bA^1}$. We will be looking at functors $\CAlg_R \rightarrow \sC$ where $\sC$ is the derived $\infty$-category of some abelian category; most likely it will be $\mbf{D}(\bZ_p)$ or the $\infty$-category of spectra. These functors will land inside the more manageable category of ``derived $p$-complete objects" denoted by $\sC_p^{\comp}$ which admits a completion functor $\sC \rightarrow \sC_p^{\comp}$, an exact left adjoint\footnote{Hence a localization.}. In general, the composite of these functors with the inclusion back to $\sC$ (i.e. the localization endofunctor) does not preserve colimits; equivalently $p$-complete objects are not stable under colimits; see \cite[Tag 0ARC]{stacks} for an example in the case of $\sC = \mbf{D}(\bZ_p)$. 

\sssec{} Fortunately, many of our invariants are bounded below and the $\bA^1$-localization functor is computed by a geometric realization~\eqref{eq:suslin}. 

\begin{lem}\label{lem:geom-rel} If $R$ is a connective $\bE_2$-ring, $I \subset \pi_0(R)$ a finitely generated ideal. Let $M_{\bullet}$ be a simplicial object in $R$-modules which are uniformly bounded below with respect to the standard $t$-structure. Then if each $M_n$ is $I$-complete, so is $|M_{\bullet}|$. 
\end{lem}

\begin{proof} We might as well assume that each $M_n$ is connective. We use the criterion (b)$\Rightarrow$(a) in \cite[Theorem 7.3.4.1]{SAG}. Indeed, it suffices to prove that $\mrm{Ext}^i_{\pi_0(R)}(\pi_0(R)[x^{-1}], \pi_k(|M_{\bullet}|)) = 0$ for $i = 0,1$. But then, since each $M_n$ is connective, $\pi_k(|M_{\bullet}|)$ on depends on a finite skeleton of $M_{\bullet}$ which is a finite colimit, whence $I$-complete. The result follows from the other direction of \cite[Theorem 7.3.4.1]{SAG}.
\end{proof}
%
%
% In the example of $\sC = \mbf{D}(\bZ_p)$ we will be considering derived $p$-complete objects in the sense of \cite[Tag 091S]{stacks} (which is a special case of \cite[Definition 7.3.1.1]{SAG}). Just as in \cite{BMS2}, functors should be $p$-completed. 

In this light, our $L_{\bA^1}$-functor will regarded as in the previous section. In other words we are still studying the effects of the endofunctor
\[
L_{\bA^1}: \PSh(\CAlg_R^{\op},\sC) \rightarrow \PSh(\CAlg_R^{\op},\sC),
\]
and indicate so when an equivalence is only known after $p$-completion by writing $L_{\bA^1}(X) \simeq_{p} L_{\bA^1}(Y)$.

%Note that this is computed by applying the completion functor to the formula~\eqref{eq:suslin}. Hence, for the rest of this section, we will (somewhat abusively) write $L_{\bA^1}$ for this functor without further notice.
%
%\begin{rem} \label{rem:does-not} If $\sC = \mbf{D}(\bZ_p)$, we note that derived $p$-complete objects are not, in general, stable under colimits \cite[Tag 0ARC]{stacks}, hence we do not know that the geometric realization in~\eqref{eq:suslin} is $p$-complete even if the terms are $p$-complete.
%\end{rem}

\ssec{Topological cyclic homology} With this technical discussion out of the way, we prove

\begin{thm} \label{thm:tc-zero} If $R \in \CAlg_{\bF_p}$, then $(L_{\bA^1}TC)(R) \simeq 0$.
\end{thm}

This implies, in particular, that the functor
\[
L_{\bA^1}TC:\CAlg_R \rightarrow \Spt,
\]
is zero for any $R \in  \CAlg_{\bF_p}$. 

\sssec{} We begin with some preliminaries. Let $R$ be a smooth (or, more generally, quasisyntomic \cite[Definition 4.10]{BMS2}) $k$-algebra where $k$ is a perfect field of characteristic $p$. Then, Bhatt-Morrow-Scholze constructed in \cite{BMS2} complete, multiplicative descending filtrations
\begin{equation} \label{eq:tp}
\mrm{Fil}_{BMS}^{\ast}TC(R) \rightarrow TC(R) \qquad \mrm{Fil}_{BMS}^{\ast}TC^-(R) \rightarrow TC^-(R) \qquad \mrm{Fil}_{BMS}^{\ast}TP(R) \rightarrow TP(R),
\end{equation} 
and also identified the associated graded pieces. 

\sssec{} Since we have restricted ourselves to characteristic $p$ (as opposed to mixed characteristic settings), the answers are particularly nice. For example we have \cite[Theorems 1.10, 1.12]{BMS2}
\[
\mrm{gr}^0_{BMS}TC^-(R) \simeq R\Gamma_{\mrm{crys}}(R/W(k)),
\]
which periodizes in $TP$ to
\[
\mrm{gr}^q_{BMS}TP(R) \simeq R\Gamma_{\mrm{crys}}(R/W(k))[2q].
\]
Here $R\Gamma_{\mrm{crys}}(R/W(k))$ is the (object in $\mbf{D}(W(k))$ computing) crystalline cohomology of $R$; note that there is a canonical representative of this object in the $\mbf{D}(W(k))$ given by the de Rham-Witt complex of Bloch-Deligne-Illusie \cite{illusie-drw}:
\[
W\Omega^{\bullet}_R = W\Omega^0_R \rightarrow W\Omega^1_R \rightarrow \cdots.
\]
\sssec{} Now, we extend the de Rham-Witt complex (and hence also crystalline cohomology) to an arbitrary animated ring by left Kan extension, whence we consider the object
\[
LW\Omega^{\bullet} \in \PSh(\CAlg_R^{\op},\mbf{D}(W(k))_p^\comp),
\]
and the localization thereof. First, we study the analogous question for the derived de Rham complex:
\[
L\Omega^{\bullet} \in \PSh(\CAlg_R^{\op},\mbf{D}(k)).
\]

\begin{lem} \label{lem:contractible} Let $k$ be a perfect field of characteristic $p$, then $L_{\bA^1}L\Omega^{\bullet} \simeq 0$.
\end{lem}

\begin{proof} Since $L_{\bA^1}L\Omega^{\bullet}(R)$ is a module over $L_{\bA^1}L\Omega^{\bullet}(k)$, it suffices to prove the following result:
\[
(L_{\bA^1}L\Omega^{\bullet})(k) \simeq \colim_{\Delta^{\op}} \Omega^{\bullet}_{k[\Delta^{\bullet}]} \simeq 0.
\]
Indeed, the above object is a module over $H^0(L_{\bA^1}\Omega^{\bullet}(k))$ which is the coequalizer of
\begin{equation} \label{eq:dr}
H^0(\Omega^{\bullet}_{k[T]}) \simeq k[T^p] \rightrightarrows H^0(\Omega^{\bullet}_{k}) \simeq k,
\end{equation}
where one of the maps sends $T^p$ to $0$ and the other to $1$. Therefore the coequalizer is the zero ring.
\end{proof}

\begin{rem} \label{rem:char0} The vanishing phenomenon described in this paper can be attributed to the fact that $H^0(\Omega^{\bullet}_{k[T]})$ is rather large in characteristic $p$. In contrast, in the presence of $\bA^1$-invariance, i.e. in characteristic zero, this group is just the base field $k$. In this case, the equalizer~\eqref{eq:dr} reads as
\[
\id,\id: k \rightrightarrows k
\]
\end{rem}

\sssec{} This gives the next vanishing result for the derived de Rham Witt complex/crystalline cohomology. This vanishing is way more than we need and is of independent interest

\begin{prop} \label{lem:contractible-1} Let $k$ be a perfect field of characteristic $p$, then $L_{\bA^1}LW\Omega^{\bullet} \simeq_p 0$, i.e., is zero after $p$-completion.
\end{prop}

\begin{proof} This follows from Lemma~\ref{lem:contractible} since 
\[
(L_{\bA^1}LW\Omega^{\bullet})/p \simeq L_{\bA^1}(LW\Omega^{\bullet}/p) \simeq L_{\bA^1}L\Omega^{\bullet} \simeq 0.
\]
\end{proof}

\sssec{} We note that the arguments in Proposition~\ref{lem:contractible-1} and Lemma~\ref{lem:contractible} also prove that the non-derived version of $\bA^1$-invariant crystalline and de Rham cohomology (evaluated on a discrete commutative ring) also vanish as they are modules over their values on $\bF_p$ where the derived and non-derived versions coincide.

\sssec{} We now finish the proof of the main result of this section.

\begin{proof}[Proof of Theorem~\ref{thm:tc-zero}]  It suffices to prove that $(L_{\bA^1}TC)(\bF_p)$ is zero. We employ the BMS filtration on $TC$. First, according to \cite[Theorem 5.1(1)]{AMMN} for a smooth $\bF_p$-algebra $R$, $\mrm{Fil}_{BMS}^{\geq i}TC(R)$ is $(i\,-\,1)$-connective. Since the terms in the colimit computing $(L_{\bA^1}\mrm{Fil}^{\geq i}_{BMS}TC)(\bF_p)$ are all smooth $\bF_p$-algebras, this spectrum is again at least $(i\,-\,1)$-connective. Therefore, as connectivity tends to $\infty$, we get that $\lim_i (L_{\bA^1}\mrm{Fil}^{\geq i}_{BMS}TC)(\bF_p) \simeq 0$ and so the $L_{\bA^1}$-BMS filtration is complete on $TC(\bF_p)$. 

To leverage this fact, we note that we have a diagram of presheaves of spectra where each column is a cofiber sequence
\begin{equation} \label{eq:tc-diagram}
\begin{tikzcd}
\ar{r} \cdots & \mrm{Fil}^{\geq 2}_{BMS}TC \ar{r} \ar{d} & \mrm{Fil}^{\geq 1}_{BMS}TC \ar{d} \ar{r} & \mrm{Fil}^{\geq 0}_{BMS}TC \ar{d}\\
\ar{r} \cdots & TC \ar{r} \ar{d} & TC \ar{r} \ar{d} & TC \ar{d} \\
\ar{r} \cdots & \mrm{Fil}^{BMS}_{< 2}TC \ar{r}  & \mrm{Fil}_{< 1}^{BMS}TC \ar{r} & \mrm{Fil}_{< 0}^{BMS}TC \simeq \mrm{gr}^0_{BMS}TC\\
\end{tikzcd}
\end{equation}
which induces the same diagram after applying $L_{\bA^1}$ again with column-wise cofiber sequences. We have proved that taking limit along the top row results in a contractible spectrum. Therefore, to prove that the middle term is contractible, we only need to prove that 
\[
(L_{\bA^1}\mrm{gr}^i_{\mrm{BMS}}TC)(\bF_p) \simeq 0
\]
for all $i \geq 0$. Since the BMS filtration is multiplicative, the (presheaf of) graded $\bE_{\infty}$-algebra(s) $\oplus \mrm{gr}^i_{BMS}TC$ is naturally a (presheaf of) $\bE_{\infty}$-algebra(s) over $\mrm{gr}^0_{BMS}TC$. Ditto their $L_{\bA^1}$-versions. Hence we only need to prove that the zero-th graded pieces vanish: $(L_{\bA^1}\mrm{gr}^0_{BMS}TC)(\bF_p) \simeq 0$. 

By the Nikolaus-Scholze formula for $TC$ \cite{nikolaus-scholze}, have a cofiber sequence of presheaves of spectra:
\[
\mrm{gr}^0_{BMS}TC \rightarrow \mrm{gr}^0_{BMS}TC^{-} \xrightarrow{\mrm{can} - \phi_p} \mrm{gr}^0_{BMS}TP,
\]
where the $\mrm{can}$ map is an equivalence in this case. Therefore, after identifying both $\mrm{gr}_{BMS}^0TP$ and $\mrm{gr}_{BMS}^0TC^-$  with crystalline cohomology on smooth $k$-schemes (we only need this for polynomial rings) we get a cofiber sequence
\[
L_{\bA^1}\mrm{gr}^0_{BMS}TC \rightarrow L_{\bA^1}R\Gamma(-/W(\bF_p)) \xrightarrow{\id - \phi_p} L_{\bA^1}R\Gamma(-/W(\bF_p)).
\]
where $\phi_p$ is the composite of $\phi_p$ with the inverse of $\mrm{can}$. Evaluating the above on $\bF_p$ and using Proposition~\ref{lem:contractible-1}, we conclude the desired result after $p$-completion. Since the terms in $TC(\bF_p[\Delta^{\bullet}])$ are all uniformly bounded below, we conclude by Lemma~\ref{lem:geom-rel}, that $L_{\bA^1}TC(\bF_p)$ is already $p$-complete and hence actually zero.

\end{proof}

\ssec{Topological periodic cyclic homology} The point of going to $TC$ in the above argument is that the diagram~\eqref{eq:tc-diagram} is $\bN^{\op}$-indexed so we can do some sort of induction. This is not the case for $TP$. We use the above result to deduce vanishing for $TP$ after $p$-completion:
 
 \begin{cor} \label{lem:tp-zero} $L_{\bA^1}TP \simeq_p 0$.
 \end{cor}
 
 \begin{proof} Using the formula \cite{nikolaus-scholze}, we view $TC$ as the equalizer of $\mrm{can},\phi_p: TC^- \rightrightarrows TP$ and similarly for the $L_{\bA^1}$-local version as this functor preserves finite limits. But now Theorem~\ref{thm:tc-zero} proves that $L_{\bA^1}TC \simeq 0$, whence the maps $\mrm{can}$ and $\phi_p$ are homotopic and, in particular, $\phi_p$ is invertible since $\mrm{can}$ is by Lemma~\ref{lem:thh}. Since everything in sight is $p$-completed suffices to prove that $L_{\bA^1}TP$ is also $\bZ[\tfrac{1}{p}]$-linear, i.e., $p$ acts invertibly.
 
% then we can appeal to a result of Scholze written up in \cite{elm} and obtain an equivalence:
% \[
% L_{\bA^1}TP \simeq L_{\bA^1}TP[\tfrac{1}{p}] \simeq L_{\bA^1}R\Gamma_{\mrm{crys}}[\tfrac{1}{p},\sigma, \sigma^{-1}].
% \]
% This last spectrum is a module over $L_{\bA^1}R\Gamma_{\mrm{crys}}$ which is zero because of Proposition~\ref{lem:contractible-1}.
 
 According to \cite[Section IV.4]{nikolaus-scholze} (see also \cite[Proposition 12]{lebras}) the homotopy groups of $TC^-(\bF_p)$ and $TP(\bF_p)$ are as follows:
\[
\pi_*(TC^-(\bF_p)) = \Z_p[u,v]/(uv-p) \qquad \pi_*(TP(\bF_p)) = \bZ_p[\sigma, \sigma^{-1}].
\]
Furthermore, since $\mrm{can}(v) = \sigma^{-1}$ and $\phi_p(v) = p\sigma^{-1}$ we have the following commutative diagram of presheaves of $TC^-(\bF_p)$-modules:
\begin{equation} \label{eq:phi-can}
\begin{tikzcd}
\Sigma^{-2}TP \ar{rr}{\sigma^{-1}} &  & TP \\
\Sigma^{-2}TC^- \ar{rr}{v} \ar{u}{\mrm{can}} \ar{d}{\phi_p} & & TC^-\ar{u}{\mrm{can}} \ar{d}{\phi_p}\\
\Sigma^{-2}TP \ar{r}{\sigma^{-1}} & TP \ar{r}{p} & TP,
\end{tikzcd}
\end{equation}
But now, after applying $\bL_{\A^1}$, we see that $\mrm{can}$ and $\mrm{\phi}_p$ are invertible by the previous discussion so that the endomorphism $p$ is invertible after applying $L_{\A^1}$.

%(note again that the terms in~\eqref{eq:suslin} are modules over $TP, TC^-$ of $\bF_p$ so that the diagram~\eqref{eq:phi-can} can be made functorial in $\bF_p$-algebras)
 \end{proof}

\begin{rem} \label{rem:alternate} Alternatively, we can prove Corollary~\ref{lem:tp-zero} by showing that $L_{\bA^1}TP/p \simeq L_{\bA^1}HP \simeq 0$. The latter equivalence is just a version of periodized de Rham cohomology which vanishes in characteristic $p$ by Lemma~\ref{lem:contractible}. We thank Antieau for pointing this out.
\end{rem}

\section{Profinite vanishing over the integers} We finish off with vanishing of profinite $TC$ after $\bA^1$-localization. Let us recall what we mean by integral $TC$ and its various localizations.

\sssec{} If $E$ is a spectrum, then the profinite completion, denoted by $E^{\comp}$ is modeled as the cofiber of the map
\[
\underline{\Maps}(\bQ, E) \rightarrow E.
\]
The model for integral topological cyclic homology is given by Nikolaus-Scholze in \cite[Section II.1]{nikolaus-scholze}:
\begin{equation} \label{eq:tc-z}
TC(R) = \mathrm{Eq}(\mrm{id}, \phi^{hS^1}:TC^-(R) \rightrightarrows TP(R)^{\comp}),
\end{equation}
using the implicit identification \cite[Lemma II.4.2]{nikolaus-scholze}
\[
(X^{tS^1})^\comp_p \simeq (X^{tC_p})^{hS^1}.
\]
We will show that the $\bA^1$-localization of profinitely-completed $TC^{\comp}$ is zero.

\sssec{} Following \cite[Definition 3.13]{LT}, given a $\bZ$-linear localizing invariant, i.e., a functor $\sF:\Cat^{\perf}_{\bZ}\rightarrow \Spt$ we can define their ${\bA^1}$-localization by taking 
\[
|\sF(\sC \otimes_{\bZ} \Perf(\bZ[\Delta^{\bullet}])| =:(L_{\bA^1}\sF)(\sC).
\]
Following the conventions of the above sections, this ${\bA^1}$-localization functors are of the form
\[
L_{\bA^1}: \PSh((\Cat^{\perf}_{\bZ})^{\op},\Spt) \rightarrow \PSh((\Cat^{\perf}_{\bZ})^{\op},\Spt),
\]
and we will indicate when equivalences are only true after further profinite completion by $L_{\bA^1}(X) \simeq^{\comp} L_{\bA^1}(Y)$.

\sssec{} The only reason to extend our functors to the noncommutative world, i.e. let it take values on $\bZ$-linear categories, is to make sense of the next lemma; see also~\ref{rem:ncworld}.

%Throughout this section, we fix an implicit prime $p$ and $L_{\bA^1}$-localization functors are again $p$-completed.

\begin{lem}\label{thm:tc} The functor $L_{\bA^1}TC$ is truncating.

\end{lem}

\begin{proof} We have a fiber sequence
\[
L_{\bA^1}K^{\mrm{inv}} \rightarrow L_{\bA^1}K=:KH \rightarrow L_{\bA^1}TC.
\]
The presheaf, $KH$ is truncating by \cite[Proposition 3.14]{LT}. 

We now claim that $L_{\bA^1}K^{\mrm{inv}}$ is a truncating invariant. To begin with, the Dundas-Goodwillie-McCarthy theorem informs us that $K^{\mrm{inv}}$ is a truncating invariant. But since $\pi_0(R[t]) \simeq \pi_0(R)[t]$ for $R$ an $\bZ_p$-$\bE_1$-algebra, we see that $K^{\mrm{inv}}$ remains truncating after applying $L_{\bA^1}$ since $L_{\bA^1}K^{\mrm{inv}}(R)$ is computed as a geometric realization of $K^{\mrm{inv}}(\pi_0(R)[T_1, \cdots T_n])$. Thus we conclude that $L_{\bA^1}TC$ is truncating.
\end{proof}

%\begin{lem} \label{thm:tc} Let $R$ be a $\bZ_p$-algebra, then the map 
%\[
%(L_{\A^1}TC)(R) \rightarrow (L_{\bA^1}TC)(R[\tfrac{1}{p}])
%\]
%is an equivalence.
%\end{lem}

%\begin{proof} We have a fiber sequence
%\[
%L_{\bA^1}K^{\mrm{inv}} \rightarrow L_{\bA^1}K=:KH \rightarrow L_{\bA^1}TC.
%\]
%The presheaf, $KH$ is truncating by \cite[Proposition 3.14]{LT}. On the other hand, the Dundas-Goodwillie-McCarthy theorem informs us that $K^{\mrm{inv}}$ is a truncating invariant. But since $\pi_0(R[t]) \simeq \pi_0(R)[t]$ for $R$ an $\bZ_p$-$\bE_1$-algebra, we see that it remains truncating. Thus we conclude that $L_{\bA^1}TC$ is truncating, whence nilinvariant by \cite[Corollary 3.5]{LT}. This, in particular, implies that $L_{\bA^1}TC(\bZ/p^n) \simeq L_{\bA^1}TC(\bF_p) \simeq 0$ by Theorem~\ref{thm:tc-zero}. We conclude the result using the same argument as in \cite[Section 2.4]{bhatt-clausen-mathew} or \cite[Section 3.1]{land-meier-tamme}
%\end{proof}

\begin{rem} \label{rem:tc} Lemma~\ref{thm:tc} gives a structural explanation to the non-$\bA^1$-invariance of $TC$. Indeed, if it were, then $TC \simeq L_{\bA^1}TC$ and so it would be nilinvariant over any base ring. This is certainly not true over $\bF_p$, by looking at the example $\bF_p[x]/(x)^2 \rightarrow \bF_p$ and the calculations of \cite{hm-poly, speirs-poly}. The non-nilinvariance of $TC$ is indeed one of the main desiderata for its invention --- as an approximation to the non-nilinvariant part of algebraic $K$-theory.

%We note that if $R$ is a ring of characteristic zero, $L_{\bA^1}TC$ of it is not necessarily zero. Indeed, if $R \oplus M \rightarrow R$ is a square-zero extension of $R$, then $R\oplus M[\Delta^{\bullet}] \rightarrow R[\Delta^{\bullet}]$ is a morphism of simplicial rings whose terms are square-zero extension. Therefore we have an equivalence on relative groups: $L_{\bA^1}TC(R \oplus M, M) \simeq KH(R \oplus M, M)$ and examples abound 
\end{rem}

\ssec{Integral topological cyclic homology} Finally:

\begin{thm} \label{thm:vanish} For any animated ring $R$, $(L_{\bA^1}TC^{\comp})(R) \simeq 0$. Therefore, the rationalization map $TC \rightarrow TC_{\bQ}$ is an $L_{\bA^1}$-equivalence.
\end{thm}
\begin{proof} As in previous arguments, we reduce to the case of $R = \bZ$. We note that the terms $TC^{\comp}(\bZ[\Delta^{\bullet}])$ are uniformly bounded below: since $TC_p^\comp$ converts $p$-adic equivalences to $p$-adic equivalences we need only prove that the terms of $TC^{\comp}_p(\bZ_p[\Delta^{\bullet}])$ are uniformly bounded below for each $p$. But this follows, for example, from the fact that $p$-complete $TC$ converts a cyclotomic spectrum whose underlying spectrum is connective to a $-1$-connective spectrum by \cite[Lemma 2.5]{CMM}.

Hence, after Lemma~\ref{lem:geom-rel}, we need only prove that $(L_{\bA^1}TC^{\comp})(\bZ) \simeq^{\comp} 0$. It then suffices to prove that $(L_{\bA^1}TC^{\comp}_p)(\bZ) \simeq_p 0$ is zero for all prime numbers $p$. Hence we need only prove that $(L_{\bA^1}TC^{\comp}_p)(\bZ) \simeq_p (L_{\bA^1}TC^{\comp}_p)(\bZ_p) \simeq 0$. To this end, we claim:
\begin{itemize}
\item the canonical map
\[
L_{\bA^1}TC^{\comp}_p(\bZ_p) \rightarrow \lim_s L_{\bA^1}TC^{\comp}_p(\bZ_p/p^s)
\]
is an equivalence.
\end{itemize}
To see that the claim implies the desired vanishing, note that since $L_{\bA^1}TC$ is truncating, it is nilinvariant by \cite[Corollary 3.5]{LT}. Therefore the limit above is stabilizes as $L_{\bA^1}TC^{\comp}_p(\bF_p$) which is zero by Theorem~\ref{thm:tc-zero}.

To prove the desired claim, we note that the limit above commutes with $L_{\bA^1}$ since the terms in $TC^{\comp}_p(\bZ_p/p^s[\Delta^{\bullet}])$ are uniformly bounded below and geometric realization behaves as a finite colimit in a range of degrees (just like the argument in Lemma~\ref{lem:geom-rel}). Thus it suffices to prove that for each $n \geq 0$, the map
\[
TC(\bZ_p[T_1, \cdots, T_n]) \rightarrow \lim_s TC(\bZ_p/p^s[T_1, \cdots, T_n]) 
\]
is a $p$-adic equivalence. Since $p$-adic $TC$ preserves $p$-adic equivalences, we may $p$-adically complete the rings inside and prove that the map
\[
TC(\bZ_p[T_1, \cdots, T_n]^{\comp}_p) \rightarrow  \lim_s TC(\bZ_p/p^s[T_1, \cdots, T_n])
\]
is a $p$-adic equivalence. This then follows by continuity of $THH$ as in \cite[Proposition 5.4]{CMM}, noting that $\bZ_p[T_1, \cdots, T_n]^{\comp}_p/p \simeq \bF_p[T_1, \cdots, T_n]$ is $F$-finite since it is finite type over $\bF_p$ and the continuity of $TC$ as in \cite[Remark 2.8]{CMM}.

%Using Lemma~\ref{thm:tc}, it suffices to prove that  $(L_{\bA^1}TC^{\comp}_p)(\bQ_p) \simeq_p 0$. The latter is obtained by evaluating the cofiber sequence
%\[
%L_{\bA^1}TC^\comp_p  \rightarrow L_{\bA^1}TC^{-\comp}_p \xrightarrow{\id - \phi_p} L_{\bA^1}TP^\comp_p
%\]
%at $\bQ_p$. But now, we note that the latter two terms vanish at rings where $p$ is invertible. Indeed:
%\begin{enumerate}
%\item we have that $TC^{-}(R) \simeq \lim_{BS^1} THH(R)$ hence is $(p)$-local since $(p)$-local objects are closed under inverse limits; so its $p$-completion vanishes;
%\item Since, $TP(R)^{\comp}_p \simeq (THH(R)^{tC_p})^{tS^1}$ we note that $THH(R)^{tC_p}$ already vanishes by \cite[Lemma I.2.9]{nikolaus-scholze}.
%\end{enumerate}
For the last statement we look at the fracture square (which is cartesian)
\[
\begin{tikzcd}
TC \ar{r} \ar{d} & TC_{\bQ}\ar{d}\\
TC^{\comp} \ar{r} & TC^{\comp}_{\bQ}.
\end{tikzcd}
\]
We have proved that, after applying $L_{\bA^1}$, the bottom left corner is zero. This means that the the bottom right corner is zero as well after applying $L_{\bA^1}$ since it is a ring admitting a ring map from the zero ring. Hence the bottom map is an equivalence after applying $L_{\bA^1}$, whence the top map is an equivalence after applying $L_{\bA^1}$ as well since $L_{\bA^1}$ preserves finite limits.
\end{proof}

\sssec{Rational situation} Let us consider these functors on $\CAlg_{\bQ}$. In this situation~\eqref{eq:tc-z} tells us that $TC \simeq TC^- = HC^-$ since $TP^{\comp} \simeq 0$. By the Geller-Weibel theorem, Corollary~\ref{cor:gw}, we further have that $L_{\bA^1}HC^- \simeq HP$. Hence we get that $L_{\bA^1}TC \simeq HP$ which is a non-zero, but familiar, invariant.

\sssec{}\label{rem:ncworld} We further note that our arguments also show that, as a $\bZ$-linear localizing invariant, $L_{\bA^1}TC$ vanishes after profinite completion. Hence we obtain an analogous ``purely rational" result for $L_{\bA^1}TC$ regarded in the noncommutative setting.

\sssec{}\label{nc-a1} In the proof of Theorem~\ref{thm:vanish}, we can avoid continuity results, by exhibiting a ``Gysin sequence" 
\[
L_{\bA^1}TC(\bF_p) \rightarrow L_{\bA^1}TC(\bZ_p) \rightarrow L_{\bA^1}TC(\bQ_p),
\]
and noting that the last term vanishes after profinite completion. Gysin sequences in the noncommutative world appears to interact well with $\bA^1$-invariance as indicated by the work of Tabuada and Van den Bergh \cite{tabuada-vdb} in geometric settings; the author thanks Mathew for pointing this out. We are working on a sequel establishing these Gysin sequences in the $\bA^1$-invariant, noncommutative world.

\bibliographystyle{alphamod}

\let\mathbb=\mathbf

{\small
\bibliography{references}
}

\parskip 0pt

\end{document}